\documentclass[11pt]{amsart}

\usepackage{amssymb,amsmath}
\usepackage{bbm}
\usepackage{a4wide}

\newtheorem{theorem}{Theorem}
\newtheorem{prop}{Proposition}

\newtheorem{coro}{Corollary}
\newtheorem{fact}{Fact}
 
{
\theoremstyle{definition}

\newtheorem{example}{Example}
}

\newcommand{\ts}{\hspace{0.5pt}}
\newcommand{\BC}{\mathbb{C}\ts}
\newcommand{\BR}{\mathbb{R}\ts}
\newcommand{\BQ}{\mathbb{Q}\ts}
\newcommand{\BZ}{\mathbb{Z}}
\newcommand{\BN}{\mathbb{N}}

\newcommand{\gG}{\varGamma} 
\newcommand{\gL}{\varLambda}

\newcommand{\sph}{\mathbb{S}}
\newcommand{\CC}{\mathcal{C}}
\newcommand{\CO}{\mathcal{O}}
\newcommand{\CP}{\mathcal{P}}

\newcommand{\ii}{\mathrm{i}}

\DeclareMathOperator{\Sim}{sim }

\DeclareMathOperator{\mul}{MR}
\DeclareMathOperator{\N}{N}

\DeclareMathOperator{\PSL}{PSL}

\DeclareMathOperator{\den}{den}

\DeclareMathOperator{\SOS}{SOS}

\newcommand{\leg}[2]{\bigl(\!\frac{#1}{#2}\!\bigr)}

\begin{document}

\title[Similar sublattices]
{Similar sublattices of planar lattices}

\author{Michael Baake}
\address{Fakult\"at f\"ur Mathematik, Universit\"at Bielefeld, 
Box 100131, 33501 Bielefeld, Germany}
\email{$\{$mbaake,pzeiner$\}$@math.uni-bielefeld.de}

\author{Rudolf Scharlau}
\address{Fakult\"at f\"ur Mathematik, 
Universit\"at Dortmund, 44221 Dortmund, Germany}
\email{Rudolf.Scharlau@math.uni-dortmund.de}

\author{Peter Zeiner}

\begin{abstract} 
  The similar sublattices of a planar lattice can be classified via
  its multiplier ring. The latter is the ring of rational integers in
  the generic case, and an order in an imaginary quadratic field
  otherwise. Several classes of examples are discussed, with special
  emphasis on concrete results. In particular, we derive Dirichlet
  series generating functions for the number of distinct similar
  sublattices of a given index, and relate them to various
  zeta functions of orders in imaginary quadratic fields.
\end{abstract}

\maketitle


\section{Introduction}

Lattices in $d$-space (by which we mean co-compact discrete subgroups
of $\BR^{d}$) are important objects with increasingly many
applications throughout mathematics and various applied sciences; see
\cite{CS} for a comprehensive study.  Among the sublattices of a
lattice $\gG \subset \BR^{d}$ are various interesting special classes,
such as similar sublattices (SSL) or coincidence site lattices (CSL);
see \cite{BM1, BM2, B} and references therein.
Their classification has important applications in crystallography,
materials science and coding theory, but is also interesting in its
own right. Here, we look at the special case of planar lattices and
derive a rather complete picture of their SSLs by using a suitable blend
of well-known results from quadratic forms, imaginary quadratic number
fields, complex multiplication and zeta functions. For known results
on the related case of planar $\BZ$-modules (in general
non-discrete), we refer to
\cite{PBR, BG, GB}.

The classification of similar sublattices is closely related to that
of coincidence sublattices, and analogously for modules, via the
underlying (generalised) symmetry groups \cite{GB, G, H, Z}. We will
thus use a formulation via the (orientation preserving) similarity
mappings of a lattice into itself, which form a ring in our
case. Beyond the planar situation, various results are known in $3$-
and $4$-space (via quaternions; see \cite{BM2, CRS, BHM, BHGZ, Z}).
General results are still sparse and restricted to rather
special cases; see \cite{CRS, H} and references therein.

In this article, we use complex numbers throughout, with \cite{Cox}
being one of our main references. For completeness and readability, we
give a brief account of the setting in Section~\ref{prelim}, followed
by a section on Dirichlet series generating functions in this
context. Section~\ref{ssl-pid} establishes the link between SSLs and
principal ideals, which is then explored in the remaining sections with
examples of increasing complexity.

\section{General setting and basic tools}\label{prelim}

Since we only consider planar lattices in this paper, we employ
complex numbers.  Two planar lattices $\gG\subset\BC$ and
$\gG{\ts}'\subset\BC$ are called (properly) \emph{similar} (or
\emph{complex homothetic}), written as $\gG \sim \gG{\ts}'$, when
$\gG{\ts}' = a\gG$ for some nonzero $a\in\BC$.  Similarity is an
equivalence relation, and we denote the equivalence class of a lattice
$\gG$ by $\Sim (\gG)$. More generally, one can (and should) also
consider orientation reversing similarities, then defining similar
lattices in the wider sense. In this paper, apart from some brief
comments, we restrict ourselves to orientation preserving mappings.

Each planar lattice can be written as the integer span of two nonzero
complex numbers $u,v$, denoted as $\gG=\langle u,v\rangle^{}_{\BZ}$,
where the ratio $v/u$ is a number in the open upper half-plane (and
thus not real).  This has an interesting and well-known consequence,
which follows from a multiplication by $1/u$.

\begin{fact} \label{fact-one}
  Each planar lattice is similar to a lattice of the form $\gG_\tau :=
  \langle 1,\tau\rangle_{\BZ}$, where $\tau$ is a complex number in
  the open upper half-plane $H:= \{z\in\BC\mid \mathrm{Im} (z)>0 \}$.
  \qed
\end{fact}

One can further restrict $\tau$ to the region given by the conditions
$\lvert \tau\rvert\ge 1$, $\lvert\tau\pm 1\rvert\ge \lvert\tau\rvert$,
compare \cite[Fig.~2.1 and Thm.~2.2]{Apo-2}, which define a
fundamental domain for the action of the modular group $\PSL
(2,\BZ)$. In this sense, knowing the similar sublattices for all
lattices $\gG_\tau$ with $\tau$ in this region is sufficient to solve
the classification problem.

Given a planar lattice $\gG\subset\BC$, let us consider the set
\begin{equation} \label{def-mul}
    \mul (\gG) \, := \, \{ a\in\BC\mid a\gG\subset\gG\}\, ,
\end{equation}    
which will be the central object for the study of planar SSLs
below. Clearly, $\mul(\gG)$ is closed under addition and
multiplication and contains $1$, so it is a ring (a subring of
$\BC$). This ring is called the \emph{multiplier ring} of $\gG$. In
particular, it always contains $\BZ$ as a subring.  For the further
analysis of $\mul(\gG)$, we recall the following concepts from
elementary algebraic number theory (see \cite{BS,Marcus} for details).

\begin{fact} \label{alg-number}
 {}For a complex number $c$, the following properties are
   equivalent:
\begin{itemize}
\item[\rm (i)] There exists a finitely generated additive subgroup 
  $M$ of\/ $\BC$ with $c M \subset M$;
\item[\rm (ii)] The number $c$ is a root of a monic polynomial with 
  integer coefficients.
\end{itemize}
Such a number is called an \emph{algebraic integer}.
\qed
\end{fact}

{}For instance, the golden ratio $(\sqrt{5} + 1)/2$ is an algebraic
integer, since it is a root of $x^2-x-1$. Clearly, an algebraic
integer is algebraic over $\BQ$ (in the sense of field extensions).
Notice that it is not required, but is a consequence of (ii), that the
minimal polynomial of an algebraic integer has integral coefficients.
Notice also that the group $M$ in (i) need not be a lattice, though it
is isomorphic to $\BZ^n$ as a group, for some $n \in \BN$. Assuming
(i), the polynomial equation of (ii) can be obtained from a matrix
representation of the linear map induced by $c$ on the rational vector
space generated by $M$.  For the converse implication, one observes
that the subgroup $M$ generated by $1,c,c^2,\dots,c^{n-1}$, where $n$
is the degree of the assumed polynomial, is mapped into itself by $c$,
since $c\cdot c^{n-1} = -m^{}_{n-1}c^{n-1}-\dots-m^{}_{0}$ for
appropriate integers $m^{}_0,\dots,m^{}_{n-1}$.

A subring $\CO$ of $\BC$ is called an \emph{order} if it is finitely
generated as a group. All elements of an order are algebraic integers
(take $M=\CO$ in Fact~\ref{alg-number}). The quotient field $K$ of
$\CO$ then is a number field, meaning a finite extension of $\BQ$.
Usually, one starts with $K$ and speaks of an \emph{order in $K$}. The
set of all algebraic integers in a given number field $K$ is also an
order, the \emph{maximal order} of $K$, denoted by $\CO_K$.

\smallskip

Let us return to the discussion of the multiplier ring $\mul(\gG)$, as
defined in \eqref{def-mul}.  It is clear that all elements in this
ring are algebraic integers (take $M=\gG$ in
Fact~\ref{alg-number}). Two lattices which are similar have the same
multiplier ring, because the multiplication in $\BC$ is
commutative. By Fact~\ref{fact-one}, it is thus sufficient to restrict
to lattices of the shape $\gG_\tau$, with $\tau\in H$. A planar
lattice $\gG$ is called \emph{generic} when $\mul (\gG) = \BZ$, and
\emph{non-generic} otherwise.  The following determination of $\mul
(\gG)$ in the non-generic case (which is the one we are mainly
interested in) is well-known from the theory of elliptic functions;
for convenience of the reader, we recall the result in some detail,
since it is fundamental for everything that follows in this paper.

\begin{prop}\label{nongenmul}
  If $\gG$ is a non-generic planar lattice, its multiplier ring\/
  $\mul(\gG)$ is an order in an imaginary quadratic field. Explicitly,
  if $\gG \in \Sim \bigl( \langle 1, \tau \rangle^{}_\BZ \bigr)$ with
  $\tau \in \BC \setminus \BR $ is non-generic, the number $\tau$ is
  algebraic of degree $2$ over $\ts\BQ$, and one has
\[
    \mul ( \gG) = \langle 1, s \tau \rangle_\BZ
\]
for an appropriate integer $s$.
\end{prop}

\begin{proof}
  As $\mul(\gG)$ is the same for all elements of $\Sim(\gG)$,
  let $\gG = \langle 1, \tau \rangle_\BZ$ be non-generic
  and consider an element $a
  \in \mul(\gG) \setminus \BZ$, which exists by assumption. By
  Fact~\ref{alg-number}, $a$ is an algebraic integer. To expand on
  this, observe that $a = a\cdot 1 \in \gG$, so $a=u+v\tau$ for some
  $u,v \in \BZ$ with $v\neq 0$. Moreover, $a\cdot \tau = u\tau + v
  \tau^2 \in \gG$ implies $u \tau + v \tau^2 = k + \ell\tau$ for some
  $k,\ell \in \BZ$. This gives a quadratic equation $v \tau ^2 +
  (u-\ell)\tau - k = 0$ over $\BZ$ (and $\BQ$) for $\tau$,
  which is thus algebraic.

  Slightly changing the notation, there is then an equation
\[
   s\tau^2 + p \tau +q \, = \, 0 \, , \quad \text{with }
   s,p,q \in \BZ \ts ,\; s>0 \ts , \text{ and } \gcd(s,p,q)=1 \ts ,
\]
where $s,p,q$ are uniquely determined by $\tau$.  Lemma 1 in
\cite[Kap.~2, \S 7.4]{BS} (derived from similar, easy computations)
now shows that $\mul (\gG)$ is as claimed in the proposition. In
particular, it is itself a planar lattice, and thus an order in the
quadratic field $\BQ (\tau)$.
\end{proof}

If, in the above proof, one writes $\tau = \alpha + \ii \beta$ with
$\alpha,\beta\in\BR$ and $\beta>0$ (so that $\tau \in H$), the
non-genericity of $\gG_\tau$ leads to an explicit necessary and
sufficient criterion for $\alpha$ and $\beta$, which follows from a
straightforward calculation.

\begin{coro} \label{rational} 
  Consider $\gG_\tau$ with $\tau = \alpha + \ii \beta$, where
  $\alpha,\beta\in\BR$ and $\beta > 0$.  This lattice is non-generic
  if and only if both $\alpha$ and $\beta^{\ts 2}\!$ are rational
  numbers.  \qed
\end{coro}

Let us briefly mention that $\tau=\frac{1}{3} + \ii \beta$ defines a
lattice $\gG$ with $3\overline{\gG}\subset\gG$, which shows the
possibility of sublattices that are similar to $\gG$ in the wider
sense.  More generally, for $\tau=\alpha + \ii \beta$, this happens if
and only if $2 m \ts \alpha + n \ts (\alpha^2 + \beta^2)$ is integer
for some $m,n\in\BZ$, not both $0$. This integrality condition is
always satisfied in the non-generic case. The existence of an
orientation reversing similarity for $\gG$ does not lead to new
sublattices precisely when the symmetry group of $\gG$ contains a
reflection.  We skip further details in this direction and concentrate
on proper similarities.

When a basis $B=\{b^{}_{1},b^{}_{2}\}$ for a planar lattice
$\gG\subset\BR^2$ is chosen, we denote by $G_B=(g_{ij})$ the corresponding
\emph{Gram matrix}, where $g_{ij}$ is the Euclidean inner product of
$b_i$ and $b_j$. A Gram matrix is called \emph{rational} when some
$0\neq \alpha\in\BR$ exists such that $\alpha\ts G_B$ has rational
entries only. Otherwise, it is called \emph{irrational}. The
rationality or irrationality of the Gram matrix (in this sense) is not
affected by the choice of the basis, and is shared by all lattices
similar to $\gG$.

\begin{coro} \label{Gram-one}
  Let $\gG$ be a planar lattice, with basis $B$ and associated
  Gram matrix $G_B$. The condition of Corollary~$\ref{rational}$ is
  then equivalent to $G_B$ being rational. This condition is
  independent of the actual choice of basis. 
  \qed
\end{coro}

Closely related to the (properly) similar sublattices of a lattice
$\gG$ is the corresponding set of orientation preserving (linear)
similarity isometries, defined as
\begin{equation} \label{eq:def-SOS}
    \SOS (\gG) = \{ z\in\sph^1 \mid \alpha\ts z \gG \subset \gG
    \text{ for some $\alpha > 0$} \}\ts .
\end{equation}
It is immediate that $\SOS (\gG)$ is a subgroup of $\sph^1$. Its
elements are referred to as the \emph{special orthogonal similarities}
(SOS) of $\gG$, although, strictly speaking, we consider only the
rotational parts of the actual similarities here. Note that the latter 
only form a monoid, which was investigated in some detail in \cite{BM1};
see also \cite{GB,G} and references therein.

\begin{theorem}\label{sos-group}
  Let $\gG$ be a planar lattice. If it is generic, it has multiplier
  ring $\mul (\gG) = \BZ$ and $\SOS$-group $\SOS (\gG) = \{ \pm 1 \}
  \simeq C_2$. Otherwise, one has $\SOS (\gG) =
  \bigl\{ \frac{w}{\lvert w \rvert} \;\big|\; 0\ne w\in \CO \bigr\}$,
  where $\mul (\gG) = \CO = \mul (\CO)$ is an order in an imaginary
  quadratic number field $K$. Its explicit form follows from
  Proposition~$\ref{nongenmul}$.  Moreover, one has
\[
   \SOS (\gG) = \SOS (\CO) = \SOS (\CO_{K})
   = \bigl\{ \frac{w}{\lvert w \rvert} \;\big|\; 
     0\ne w\in \CO_{K} \bigr\} ,
\]
where $\CO_{K}$ is the maximal order of $K$
and contains $\CO$, and $\SOS (\gG)$ is constant on $\Sim (\gG)$.
\end{theorem}
\begin{proof} In view of Proposition~\ref{nongenmul}, the claims
  follow from the observation that the SOS-group precisely consists of
  the directions $w/\lvert w\rvert$ with $w\ne 0$ in the multiplier
  ring of $\gG$, expressed as numbers on the unit circle. Clearly,
  $\CO$ is also its own multiplier ring, and every direction in $\CO$
  is a direction in $\CO_{K}$. On the other hand, $\CO$ has finite
  index in $\CO_{K}$, say $n$, so that $nz\in\CO$ for all
  $z\in\CO_{K}$, and the last claim follows.
\end{proof}

Let us mention in passing that $\SOS (\gG)$ remains unchanged for each
lattice that is commensurate with $\gG$ (meaning that there is a
common sublattice), but also for all elements of $\Sim(\gG)$ (and thus
for all lattices commensurate with any of the latter).  This is a
special feature of the planar situation (and trivially true for
$d=1$), but does not hold in higher dimensions, as one loses
commutativity of the special orthogonal group.

\begin{example}[{$\SOS (\BZ[\ii])$ and 
$\SOS (\BZ[\frac{1+\ii\sqrt{3}}{2}])$}]\label{ex:square-sos}
As $\BZ[\ii]$ is a principal ideal domain (and even Euclidean), its
arithmetic can be used to derive $G=\SOS (\BZ[\ii])$ explicitly. If
$z=\frac{w}{\lvert w\rvert}\in G$, then so is $z^2 = w^{2}/\lvert w
\rvert^2$. Using the unique prime decomposition \cite{HW} up to units
in $\BZ[\ii]$ together with the fact that $\lvert w \rvert^2 \in \BN$,
one finds
\[
      z^2 = \varepsilon \prod_{p\equiv 1\; (4)} 
      \Bigl( \frac{\,\omega^{}_{p}\,}{\overline{\omega^{}_{p}}} 
      \Bigr)^{n^{}_{p}} 
      = \varepsilon \prod_{p\equiv 1\; (4)} 
      \Bigl( \frac{\,\omega^{2}_{p}\,}{p} \Bigr)^{n^{}_{p}} ,
\]
where $\varepsilon = \ii^k$ with $k\in\{0,1,2,3\}$ is a unit in
$\BZ[\ii]$ and the product runs over the splitting primes of the
field extension $\BQ(\ii) / \BQ$. Here, all $n^{}_{p}\in\BZ$, at
most finitely many of them non-zero, and $p=\omega^{}_{p} \ts
\overline{\omega^{}_{p}}$ is the splitting of $p\equiv 1 \bmod 4$ into
two non-associate Gaussian primes; for details of this derivation, we
refer to \cite{PBR,B} and references therein.  Clearly, one then has
\[
    z = \Bigl( \frac{1+\ii}{\sqrt{2}} \Bigr)^{\ell} \prod_{p\equiv 1\; (4)} 
        \Bigl( \frac{\,\omega^{}_{p}\,}{\sqrt{p}} \Bigr)^{n^{}_{p}} 
\]
for some $\ell\in\{0,1,\ldots,7\}$ and the $n^{}_{p}\in\BZ$ with the
restrictions as above. Noting that $(1+\ii)/\sqrt{2}$ is a primitive
$8$th root of unity, one concludes $\SOS (\BZ[\ii]) \simeq C_{8}
\times \BZ^{(\aleph^{}_{0})}$. Explicit choices of the corresponding
generators can be read from the previous formula.

An analogous result holds for the triangular lattice, where the
$\SOS$-group is $C_{12}\times\BZ^{(\aleph^{}_{0})}$, with a primitive
$12$th root of unity as generator for the cyclic group $C_{12}$ and
$\omega^{}_{p}/\sqrt{p}$ with $p\equiv 1 \bmod 3$ as generators for
the infinite cyclic groups, where $\omega^{}_{p}$ is a (complex)
Eisenstein prime in the Euclidean ring $\BZ[\frac{1+\ii\sqrt{3}}{2}]$;
see \cite{HW} for background.

\end{example}

\section{Generating functions}

If $\gG$ is a planar lattice, we denote the number of distinct SSLs of
$\gG$ of index $m$ by $f(m)$. The integer-valued arithmetic function
$f$ is super-multiplicative, which means that one has
$f(mn)\ge f(m)\ts f(n)$ for coprime $m,n\in\BN$, see \cite{BHM} and
references therein for details. An example for genuine
super-multiplicativity is given by the rectangular lattice $\langle
1,\tau\rangle^{}_{\BZ}$ with $\tau=3\ts\ii/2$; further examples will
follow below.

In many interesting cases,
however, $f$ is a multiplicative function, which motivates the use of
Dirichlet series as their generating functions. We thus define
\begin{equation} \label{diri-one}
   D^{}_{\gG} (s) \, := \,
   \sum_{m=1}^{\infty} \frac{f(m)}{m^s} .
\end{equation}
As $[\gG : m\gG ] = m^2$, a lower bound for $f(m)$ is given by the
function that takes the value $1$ on all squares of $\BN$ and the
value $0$ otherwise. This lower bound gives the Dirichlet series of
the function $\zeta(2s)$, which converges absolutely for all $s$ with
$\mathrm{Re} (s) > \frac{1}{2}$. An upper bound is the number of
\emph{all} sublattices of $\gG$ of index $m$, which is given by the
divisor function $\sigma^{}_{1} (m) = \sum_{d|m} d$; see
\cite[Appendix]{B} or \cite[p.~99, Lemma~2]{S}.  
It defines the Dirichlet series of $\zeta(s)
\zeta(s-1)$, with absolute convergence for all $s$ with $\mathrm{Re}
(s) > 2$.  This implies that all Dirichlet series $D^{}_{\gG} (s)$ of
planar lattices converge absolutely at least in the open right
half-plane $\{ s\in \BC\mid \mathrm{Re} (s) > 2 \}$.

Recall from \cite{BHM} that a sublattice $\gL$ of $\gG$ is called
\emph{primitive} in $\gG$ when $x\gL\subset\gG$ with
$x\in\BQ$ implies $x\in\BZ$.  It is advantageous to distinguish SSLs
that are primitive from those that are not. In fact, each sublattice
of $\gG$ can uniquely be written as $k\gL$ with $k\in\BN$ and $\gL$ a
primitive sublattice. If we count the number of primitive SSLs of
$\gG$ by the function $f^{\sf pr} (m)$ and define $D^{\sf pr}_{\gG}
(s) := \sum_{m=1}^{\infty} \frac{f^{\sf pr}(m)}{m^s}$ in analogy to
\eqref{diri-one}, it is clear that one always has the relation
\begin{equation} \label{diri-two}
   D^{\vphantom{p}}_{\gG} (s) \, = \, 
       \zeta(2s)\, D^{\sf pr}_{\gG} (s) \ts .
\end{equation}
The determination of the generating function is thus reduced to
finding its primitive part,  the Dirichlet series $D^{\sf pr}_{\gG} (s)$.
\begin{fact}
  If $\gG$ is a planar lattice with generic multiplier ring,
  which is $\BZ$, one has $D^{\sf pr}_{\gG} (s) = 1$ and thus
  $D^{}_{\gG} (s) = \zeta(2s)$. 
  \qed
\end{fact}
In previous articles, the generating functions have been calculated
for a variety of examples in the plane (see \cite{BM1,BG} and
references therein) and in higher dimensions (compare
\cite{BM1,BM2,CRS,BHM}). Standard results such as Delange's Theorem
\cite[Thm.~II.15]{T} then yield the asymptotic growth of
$\sum_{m=1}^{n} f(m)$ for large $n$, which is one further benefit
of using generating functions.  It is now our aim to develop a
general approach for the calculation of the generating functions in
the planar case.

\section{Similar sublattices and principal ideals}\label{ssl-pid}

Let $\gG$ be a planar lattice with non-trivial multiplier ring $\mul
(\gG)$, which is thus an order $\CO$ in an imaginary quadratic field
$K$.  Note that $\CO$ itself is a planar lattice, and its own
multiplier ring, though it need not be similar to $\gG$ (we will see
examples for this below).  Nevertheless, the rotation symmetry group
of $\gG$ is canonically isomorphic with the unit group $\CO^\times$,
which is $C_2$, $C_4$ (when $\gG$ is similar to the standard square
lattice, $\gG\in\Sim(\BZ^2)$) or $C_6$ (when $\gG$ is similar to the
regular triangular lattice, $\gG\in\Sim (A_2)$).  Observe that the
linear mapping $z\mapsto a z$ in $\BC$ has determinant
$a\bar{a}$. Consequently, one has $[\gG : a \gG]= a \bar{a}$ for any
non-zero $a\in\CO$, by a standard argument involving areas of
fundamental domains.  In other words, $a\gG$ is an SSL of $\gG$ of index
$a\bar{a} = \N (a)$, where $\N$ denotes the field norm of $K$ and the
nontrivial Galois automorphism needed here is complex conjugation
$z\mapsto \bar{z}$.

\begin{prop}\label{prop:ideals}
  If $\gG$ is a planar lattice with multiplier ring\/ $\mul (\gG) = \CO
  \ne \BZ$, one has an index-preserving bijection between the SSLs of
  $\gG$ and the principal ideals of $\CO$. The Dirichlet series
  generating function for the number of SSLs of $\gG$ of a given index
  is thus given by the Dirichlet series for the non-zero principal
  ideals of $\CO$.
\end{prop}

\begin{proof}
  The lattice $\gG$ is similar to a lattice $\gG_{\tau}$ for some
  $\tau$ in the fundamental domain of the modular group, as discussed
  above.  By assumption and an application of
  Proposition~\ref{nongenmul}, $K=\BQ(\tau)$ is then an imaginary
  quadratic field, and the multiplier ring of both $\gG_{\tau}$ and
  $\gG$ is an order $\CO$ in $K$. Observe that $a\CO$ is a principal
  ideal of $\CO$ of index $\N(a)$.  Since $a\gG=b\gG$ for non-zero
  $a,b\in\CO$ implies $b^{-1}a\CO=\CO$, the number $b^{-1}a$ must be a
  unit in $\CO$. Conversely, any unit $\varepsilon \in\CO$ satisfies
  $\varepsilon\gG\subset\gG$.  Since $\N(\varepsilon)=1$, one actually
  has equality, which establishes the bijectivity as claimed.

The generating function then satisfies
\[
    D^{}_{\gG} (s) \, = \, \sum_{m=1}^{\infty}
    \frac{f (m)}{m^s} \, = \,
    \sum_{\substack{0\ts \neq \ts \mathfrak{a}\ts \subset\ts \CO \\
    \mathfrak{a} \text{ is principal}}}
    \frac{1}{\N(\mathfrak{a})^s} ,
\]
where $\mathfrak{a} = a\CO$ for some $a\in\CO$ when $\mathfrak{a}$ is
principal. Since $\N(\mathfrak{a}) = [\CO:\mathfrak{a}]= \N(a)$ in
this case, the second claim follows.
\end{proof}

For the remainder of the article, we will now use our approach to
treat concrete classes of examples, in increasing order of complexity.

\section{Orders of class number $1$} \label{CN1}

A particularly nice and simple situation emerges when the multiplier
ring $\CO$ of $\gG$ is a principal ideal domain (PID), or when at
least all proper ideals are principal (see below for more). In this
case, the Dirichlet series $D^{}_{\gG} (s)$ is just the zeta function
of $\CO$ itself, which is the Dirichlet series generating function for
\emph{all} non-zero ideals of $\CO$. To continue, it is easier to make
the distinction whether the order $\CO$ is maximal or not.

\subsection{Maximal orders}
Let $K$ be an imaginary quadratic field of class number $1$, with
discriminant $d_{K}$ (we follow the notation of \cite{Cox}), and let
$\CO=\CO_K$ be the maximal order of $K$, which is the ring of integers
in $K$ and a PID due to the assumption on the class number.  The following
result is classic, compare \cite[Thm.~7.30]{Cox}.
\begin{fact}\label{CNone}
  There are precisely $9$ imaginary quadratic fields with
  class number $1$, which means that their maximal orders
  are PIDs. These are the fields $K=\BQ(\omega^{}_{0})$ for
\[
   \omega^{}_{0} \, \in \, \big\{ \tfrac{1+i\sqrt{3}}{2}, i,
   \tfrac{1+i\sqrt{7}}{2}, i \sqrt{2}, \tfrac{1+i\sqrt{11}}{2},
   \tfrac{1+i\sqrt{19}}{2}, \tfrac{1+i\sqrt{43}}{2},
   \tfrac{1+i\sqrt{67}}{2}, \tfrac{1+i\sqrt{163}}{2} \big\}\, ,
\]
  which are fields of discriminant $d_{K}\in 
  \{-3,-4,-7,-8,-11,-19,-43,-67,-163\}$.
  In this formulation, the maximal order is $\CO_K =
  \BZ[\omega^{}_{0}]$, while 
  $\BQ(\omega^{}_{0}) = \BQ(\sqrt{d_{K}}\,)$. \qed
\end{fact}

The zeta function of $\CO_K$ is the Dedekind zeta function of the
quadratic field $K$. It is known \cite{Zag} to factorise as
\begin{equation} \label{quadratic-zeta}
   \zeta^{}_{K} (s) \, = \, \zeta(s) \, L(s,\chi) \, ,
\end{equation}
where $L(s,\chi)$ is the $L$-series of the nontrivial character
$\chi = \chi^{}_{d_{K}}$ of the field $K$.
The latter is a totally multiplicative arithmetic function and thus
given by $\chi^{}_{d_{K}} (1) =1$ together with its values on 
rational primes,
\[
    \chi^{}_{d_{K}} (p) = \begin{cases}
    0 \ts , & p \mid d_{K}\ts ,  \\[1mm]
    \bigl( \frac{d_{K}}{p}\bigr), & 2\ne p \nmid d_{K} \ts ,\\[1mm]
    \bigl( \frac{d_{K}}{2}\bigr), & p=2\nmid d_{K}\ts .
    \end{cases}
\]
Here, $\bigl(\frac{d_{K}}{p}\bigr)$ and $\bigl(\frac{d_{K}}{2}\bigr)$
denote the Legendre and the Kronecker symbol, the latter defined as
\[
    \left(\frac{d_{K}}{2}\right) = \begin{cases}
    1 \ts , & d_{K} \equiv 1 \; (8) \ts , \\
    -1\ts , & d_{K} \equiv 5 \; (8) \ts , \\
    0\ts , & d_{K} \equiv 0 \; (4)\ts .
    \end{cases}
\]
This permits a direct calculation of the zeta function via its Euler
product, as the character $\chi(p)$ takes only the values $0$, $-1$,
or $1$, depending on whether the rational prime $p$ ramifies, is
inert, or splits in the extension from $\BQ$ to $K$. The general
formula reads
\begin{equation}\label{quadratic-zeta-two}
      \zeta^{}_{K} (s) = \prod_{p\in\CP}  
      \frac{1}{(1-p^{-s})(1-\chi(p) p^{-s})} 
      \, = \!\!
      \prod_{\substack{p\in\CP \\ \chi(p)=0}} \! \frac{1}{1-p^{-s}} \!
      \prod_{\substack{p\in\CP \\ \chi(p)=-1}} \! \!\frac{1}{1-p^{-2s}} 
      \prod_{\substack{p\in\CP \\ \chi(p)=1}} \! \frac{1}{(1-p^{-s})^2} \ts ,
\end{equation}
where $\CP$ denotes the set of rational primes.

Let us recall that Eq.~\eqref{quadratic-zeta} implies the relation
\[
   f^{}_{K} (m) = \sum_{\ell | m} \chi^{}_{d_{K}} (\ell) 
\]
for the number of principal ideals of norm $m$ in $\CO_K$. This is
also the number of representations of $m$ by the norm form (counted
modulo the unit group of $\CO_K$), which can be proved by elementary
means as well; compare \cite[Thm.~8.3]{Zag}. Either way, one can now
calculate the contributions from primitive lattices by means of
Eq.~\eqref{diri-two}. An Euler factor that will show up repeatedly in
these zeta functions is
\begin{equation} \label{standard-factor}
   \frac{1+p^{-s}}{1-p^{-s}} \, = \, 
   1 + \frac{2}{p^s} + \frac{2}{p^{2s}}
   + \frac{2}{p^{3s}} + \ldots
\end{equation}  
The result on the generating functions now reads as follows.

\begin{prop}\label{prop:genfun}
  Let $K$ be any of the $9$ imaginary quadratic fields of
  Fact~$\ref{CNone}$, and let $p_{\sf ram}$ be its ramified prime,
  which is the unique rational prime that divides $d_{K}$. The
  Dirichlet series generating function for the number of SSLs of
  $\ts\CO_{K}$ is given by $D^{}_{\CO_{K}} (s) = \zeta^{}_{K} (s)$
  with the Dedekind zeta function of $K$ according to
  Eq.~\eqref{quadratic-zeta-two}.

Moreover, the generating function for the primitive
SSLs of $\ts\CO_{K}$ is
\[
     D^{\sf pr}_{\CO_{K}} (s) 
     = \frac{D^{}_{\CO_{K}} (s) }{\zeta(2s)}
     = (1+ p_{\sf ram}^{-s})
     \prod_{p\; \mathrm{splits}} 
     \frac{1+p^{-s}}{1-p^{-s}} \ts ,
\]
where the product runs over all rational primes $p$ that split in the
extension to $K$.  The same generating function also applies to any
planar lattice $\gG \in \Sim( \CO_K ) $.  \qed
\end{prop}

If we write $D^{\sf pr}_{\CO_{K}} (s) = \sum_{m=1}^{\infty} f^{\sf pr}
(m) \, m^{-s}$, the arithmetic function $f^{\sf pr}$ satisfies $f^{\sf
pr} (m) = 0$ for any $m\in\BN$ that is divisible by $p^{2}_{\sf ram}$
or by an inert prime. Otherwise, it takes the value $2^a$, where $a$
is the number of distinct splitting primes that divide $m$.

\begin{table}
{\large
\renewcommand{\arraystretch}{1.25}
\begin{tabular}{c|c||c|c||c|c}
$d_{K}$ &  norm form & 
$d_{K}$ &  norm form & 
$d_{K}$ & norm form  \\ \hline
$-3$   &  $x^2 + xy + y^2$ & 
$-8$   &  $x^2 + 2 y^2$ &
$-43$  &  $x^2 + xy + 11 y^2$ 
\\
$-4$   &  $x^2 + y^2$ & 
$-11$  &  $x^2 + xy + 3 y^2$ &
$-67$  &  $x^2 + xy + 17 y^2$
 \\
$-7$   &  $x^2 + xy + 2 y^2$ &
$-19$ &  $x^2 + xy +  5 y^2$ &
$-163$ & $x^2 + xy + 41 y^2$
\end{tabular}}
\bigskip
\caption{Norm forms for the $9$ maximal orders $\CO=\BZ
[\omega^{}_{0}]$ of class number $1$ in imaginary quadratic number
fields, labelled with the field discriminant $d_{K}$.}
\label{tab:max}
\end{table}

It remains to formulate a characterisation of the index spectrum and
the primitive index spectrum, meaning the integers $m$ for which
$f(m)\ne 0$ or $f^{\sf pr} (m)\ne 0$.  The result can be phrased by
means of the norm form of $\CO_{K}$, which is given in
Table~\ref{tab:max}.

\begin{coro}
  Let $\gG$ be a planar lattice with $\mul (\gG) = \CO_{K}$ for one of
  the $\ts 9$ imaginary quadratic fields $K$ of Fact~$\ref{CNone}$.
  Then, the indices of the SSLs of $\gG$ are precisely the positive
  integers that can be represented by the norm form of $K$, while
  those of the primitive SSLs are the subset of primitively
  representable integers. \qed
\end{coro}

\begin{example}[Square and triangular lattices]\label{ex:square-tri}
The square lattice $\BZ[i]$ and the triangular lattice
$\BZ[\frac{1+i\sqrt{3}}{2}]$ are the most prominent examples, and also
(up to similarity) the only ones with a larger point symmetry, as
mentioned above.  Since they have been analysed explicitly in various
other sources, see \cite{B,BM1,BG} and references therein, we omit
further details of the derivation and simply state the result. For any
lattice $\gG\in\Sim \bigl(\BZ[\ii] \bigr)$,
Proposition~\ref{prop:genfun} leads to the generating function
\begin{equation}\label{eq:squaregenfun}
     D^{\sf pr}_{\square} (s) = \sum_{m=1}^{\infty} 
    \frac{f^{\sf pr}_{\square}(m)}{m^{s}} = (1+2^{-s}) 
      \prod_{p\equiv 1 \; (4)}
      \frac{1+p^{-s}}{1-p^{-s}} \ts .
\end{equation}
Here, $f^{\sf pr}_{\square} (m) =0$ whenever $m$ is divisible by $4$
or by any prime $p\equiv 3 \bmod 4$, while one has $f^{\sf
pr}_{\square} (m) = 2^a$ otherwise, where $a$ is the number of
distinct primes $p\equiv 1 \bmod 4$ that divide $m$.

Similarly, for any $\gG\in \Sim \bigl(\BZ[\frac{1+i\sqrt{3}}{2}]
\bigr)$, one obtains
\begin{equation}\label{eq:trigenfun}
     D^{\sf pr}_{\triangle}  (s)  =  
       \sum_{m=1}^{\infty} \frac{f^{\sf pr}_{\triangle}(m)}{m^{s}}
      = (1+3^{-s}) \prod_{p\equiv 1 \; (3)} 
      \frac{1+p^{-s}}{1-p^{-s}} \ts .
\end{equation}
In variation of the previous case, one now has $f^{\sf pr}_{\triangle}
(m) =0$ for all $m$ that are divisible by $9$ or by any prime $p\equiv
2 \bmod 3$, and otherwise $f^{\sf pr}_{\triangle} (m) =2^a$, this time
with $a$ being the number of distinct primes $p\equiv 1 \bmod 3$ that
divide $m$.
\end{example}

\subsection{Non-maximal orders}
An application of the general class number formula for orders, see
\cite[Part~I, Thm.~7]{Sch123} or \cite[Thm.~7.24]{Cox}, shows that
there are precisely $4$ non-maximal orders of class number $1$ in
imaginary quadratic fields. Note, however, that a non-maximal order
$\CO$ fails to be Dedekind, hence is never a PID in the usual
sense. Here, the ideal class group only refers to the \emph{proper}
(or invertible) ideals, see \cite[\S 7]{Cox} for a nice summary. In
particular, all principal ideals are proper, wherefore we still have a
useful connection with the zeta function of $\CO$. The basic data for
our purposes are summarised in Table~\ref{tab:non-max}.

\begin{table}
{\large
\renewcommand{\arraystretch}{1.25}
\begin{tabular}{c|c|c|c|c|c}
$D$ & $K$ & $\CO$ & norm form & $p\ts | D$ & conductor \\ \hline
$-12$ & $\BQ(\ii\sqrt{3}\,)$ & $\BZ[\ii\sqrt{3}\,]$ & 
$x^2 + 3 y^2$ & $2,3$ & $2$ \\
$-16$ & $\BQ(\ii)$ & $\BZ[2\ts\ii]$ & $x^2 + 4 y^2$ & $2$ & $2$ \\
$-27$ & $\BQ(\ii\sqrt{3}\,)$ & $\BZ[\tfrac{1}{2}
                (1 + \ii \ts 3 \sqrt{3}\,)]$ & 
$x^2 + xy +  7 y^2$ & $3$ & $3$ \\
$-28$ & $\BQ(\ii\sqrt{7}\,)$ & $\BZ[\ii\sqrt{7}\,]$ & 
$x^2 + 7 y^2$ & $2,7$ & $2$
\end{tabular}}
\bigskip
\caption{Basic data for the $4$ non-maximal orders of class number $1$
in imaginary quadratic number fields, labelled with their discriminant $D$.} 
\label{tab:non-max}
\end{table}

In our present situation, it turns out that the generating function
for $\CO$ still possesses an Euler product over all primes.  This is
clear for all but finitely many primes, due to the bijection property
between ideals of $\CO$ and those of $\CO_{K}$ with norms coprime to
the conductor; see \cite[Prop.~7.20]{Cox}, and \cite[Ex.~8.8]{Zag} for
an explicit expression in terms of characters.  For the finitely many
remaining primes, namely the ones dividing the conductor, one has to
do some extra calculations, which then give the remaining Euler
factors constructively. This will be outlined in the explicit
treatment of the examples below, where we actually show this for all
primes that divide the discriminant.  As before, we focus on the
Dirichlet series for the primitive SSLs, because the others simply
follow by multiplication with $\zeta(2s)$, as in Eq.~\eqref{diri-two}.

\begin{example}[$D=-12$]\label{D=-12}
  The primes that need special attention are $p=2$ and $p=3$.  The
  quadratic form $x^2 + 3 y^2$ cannot represent $2$, while congruence
  arguments (mod $8$ and $9$) show that it cannot \emph{primitively}
  represent any integer that is divisible by $8$ or $9$. On the other
  hand, $3=0+3\ts (\pm 1)^2$ and $4=(\pm 1)^2 + 3\ts (\pm 1)^2$ are
  the only possibilities to represent $3$ and $4$, respectively.
  Counted modulo the unit group $\CO^{\times} \!\simeq C_2$, this
  amounts to a single solution for $m=3$ and to two solutions for
  $m=4$. All other primes can be extracted from the general
  formula~\eqref{quadratic-zeta}. The multiplicativity of the counting
  function (by the relation to $\CO_{K}$) is inherited for the
  combination of all primes except $p=2$. By another congruence
  argument (mod $4$), which in essence explores the different unit
  groups of $\CO$ and $\CO_{K}$, one sees that any primitive
  representation $x^2 + 3 y^2 = 4 m$ with $m$ odd can be split into
  one of $4$ and one of $m$, so that multiplicativity holds also for
  this prime factor.  Together, this results in the Dirichlet series
\[
     D_{\CO}^{\sf pr} (s) \, = \,
      \bigl( 1+ \frac{2}{4^s} \bigr) \bigl( 1 + \frac{1}{3^s} \bigr)
      \prod_{p\ts\equiv 1\; (3)} \frac{1+p^{-s}}{1-p^{-s}} .
\]     
\end{example} 

\begin{example}[$D=-16$]\label{D=-16}
  Here, the only special prime is $p=2$. When $m=x^2 + 4 y^2$ is
  divisible by $16$, congruence arguments mod $4$ and $16$ show that
  $x$ and $y$ cannot be coprime, so that no primitive solutions are
  possible then. As $2$ is not representable at all, it remains to
  count the solutions for $m=4$ and $m=8$, where one observes $4 = 0^2
  + 4\ts (\pm 1)^2$ and $8 = (\pm 2)^2 + 4\ts (\pm 1)^2$, which (again
  mod the unit group $\CO^{\times}\!\simeq C_{2}$) amounts to $1$
  resp.\ $2$ solutions. As in the previous example, the
  multiplicativity of the counting function needs to be extended,
  here to cover powers of $p=2$. It follows from a congruence
  argument mod $8$ resp.\ mod $16$. Together with the standard Euler
  factor~\eqref{standard-factor} for all other primes, one thus has
  the Dirichlet series
\[
    D_{\CO}^{\sf pr} (s) \, = \,
    \bigl( 1+ \frac{1}{4^s} + \frac{2}{8^s} \bigr) 
    \prod_{p\ts\equiv 1\; (4)} \frac{1+p^{-s}}{1-p^{-s}}.
\]

\end{example} 

\begin{example}[$D=-27$]\label{D=-27}
  Here, the kind of reasoning of the previous example has to be
  repeated for the prime $p=3$, though for a slightly more complicated
  quadratic form.  One can check that $3$ is not representable by $x^2
  + xy + 7 y^2$, while $9=1^2 + 1\cdot 1 + 7\cdot 1^2= 2^2 - 2\cdot 1
  + 7\ts (-1)^2$ and $27=4^2 + 4\cdot 1 + 7\cdot 1^2 = 1^2 - 1\cdot 2
  + 7\ts (-2)^2 = 5^2 - 5\cdot 1 + 7\ts (-1)^2$ provide a complete
  list of representatives (mod units) for the primitive
  representations of $9$ and $27$. To see that no primitive
  representation of integers of the form $81 m$ with $m\in\BZ$ exist,
  one first observes $x^2 + xy + 7 y^2 = (x+\tfrac{1}{2}\ts y)^2 +
  \tfrac{27}{4} \ts y^2$, and concludes via congruence considerations
  mod $81$.

  Moreover, a refined congruence argument (mod $27$) also shows that,
  as in the previous two examples, we get an extension of
  multiplicativity to cover contributions from powers of $p=3$.
  Invoking the standard Euler factor once more for all other primes,
  one gets
\[
    D_{\CO}^{\sf pr} (s) \, = \,
    \bigl( 1+ \frac{2}{9^s} + \frac{3}{27^s} \bigr) 
    \prod_{p\ts\equiv 1\; (3)} \frac{1+p^{-s}}{1-p^{-s}}.
\]
\
\end{example} 

\begin{example}[$D=-28$]\label{D=-28}
  In the last example of this paragraph, the primes $p=2$ and $p=7$
  need special attention, this time for the quadratic form $x^2 + 7
  y^2$.  Clearly, there is only one way (mod units) to represent $7$,
  and no primitive way to represent any integer that is divisible by
  $49$, which follows once more by a congruence argument (here,
  mod $49$).

  {}For the positive powers of the prime $2$, one quickly finds that
  $2$ and $4$ are not representable at all. The primitive
  representations of the higher powers of $2$ can be derived from the
  factorisation $2 = \pi \bar{\pi}$ with $\pi = (1+\ii\sqrt{7}\,)/2$,
  where $\pi$ is a prime in the maximal order (which is $\CO_K =
  \BZ[\pi]$), but not an element of $\CO$. Observe next that the only
  ideals of $\CO_K$ of norm $2^r$ are the principal ideals generated
  by $\pi^{\ell} \bar{\pi}^{r-\ell}$ for $0\le\ell\le r$. We need to
  select the generating elements that also lie in $\CO$ and are
  primitive there.  It is not difficult to check that this requires
  $r\ge 3$ together with either $\ell=1$ or $r-\ell = 1$.  These two
  cases are not related by units, so that always precisely two
  primitive representations (up to units) exist for $r\ge 3$.
  
  Now, one needs the identity
\[
     1 + \sum_{m\ge 3}\frac{2}{2^{ms}} \, = \,
     1 + \frac{2}{8^s} \,\frac{1}{1-2^{-s}} \, = \,
     (1-2^{1-s} + 2^{1-2s} )\,\frac{1+2^{-s}}{1-2^{-s}} \ts ,
\]
while all remaining primes work as in the previous examples. 
Here, multiplicativity of the counting function is once again
clear for all primes except $p=2$. For the latter, we observe
that an integer in $\CO_{K}$ with odd norm is automatically an
element of $\CO$, so that we can factorise any represented
integer into powers of $2$ and its odd part. Together, this yields
\[
    D_{\CO}^{\sf pr} (s) \, = \,
      \bigl( 1- \frac{2}{2^s} + \frac{2}{4^s} \bigr) 
      \bigl( 1 + \frac{1}{7^s} \bigr)
      \prod_{p\ts\equiv 1,2,4 \; (7)} \frac{1+p^{-s}}{1-p^{-s}} \ts .
\]

\end{example}

\section{Euler's convenient numbers}

Similar results can be obtained for a larger, though still
finite, list of discriminants. These are the numbers $D$ such that
every \emph{genus} of (positive definite, binary) quadratic forms of
discriminant $D$ consists of one class only. The crucial property of
such single class genera is that, for the corresponding forms, it only
depends on a congruence condition modulo $D$ whether a natural number
is represented by the form or not.  By definition, two quadratic forms
(in any number of variables) are in the same genus if they are
equivalent modulo $N$ for every modulus $N \in \BN$. In this case, the
forms have the same discriminant, and the number of classes in one
genus is thus always finite.

Here, we deal with \emph{binary} quadratic forms where the theory of
genera has several special features (and is, in fact, a well
established part of classical algebraic number theory, independent of
the general theory of quadratic forms; compare \cite{Buell,Cox,Zag}).
As before, the distinction between fundamental and non-fundamental
discriminants is relevant. For a given \emph{fundamental} discriminant
$D$, the equivalence classes of quadratic forms bijectively correspond
to the ideal classes in the maximal order $\CO_D$.  For a quick
description of the partition of classes into genera, one can take
advantage of the group structure on the set $\CC_D$ of ideal classes
of $\CO_D$:\ two ideal classes are in the same genus if they give the
same element in the factor group $\CC_D^{}/\CC_D^2$. All genera are of
the single class type if and only if $\CC_D^2$ is the trivial group,
which is tantamount to saying that the class group is a finite Abelian
$2$-group.  For \emph{non-fundamental} discriminants, there are
certain complications to this approach (which works only for
invertible ideals).  We therefore briefly summarise the main facts in
a different way, which is more suitable for our purposes.

The different genera of binary forms $q$ of some fixed discriminant
$D$ are separated by the values $m=q(x,y)$ represented by the
form. Together with an individual $m\in\BZ$ coprime to $D$, also its
whole square class in $(\BZ/D\BZ)^{\times}$ is represented by the
genus. Already one square class represented by $q$ determines the
genus of $q$.  This square class, in turn, is determined by the values
of all quadratic (or `real') characters $\chi \!:\,
(\BZ/D\BZ)^{\times} \longrightarrow \{\pm 1\}$.  Let us mention in
passing that precisely half of the elements of $(\BZ / D\BZ)^{\times}$
are represented by some form of discriminant $D$. These are the
elements of the kernel of a certain `principal' character
$\chi^{}_{\!\ts D}$; compare \cite{Buell, Cox, Zag}.

Following our earlier discussion, we are primarily interested in the
principal genus, which contains the norm form of the order $\CO_D$.
The elements of $(\BZ/D\BZ)^{\times}$ represented by this genus form
a subgroup of $(\BZ/D\BZ)^{\times}$ that contains the group of all
squares as another subgroup of index at most $2$;
compare \cite[Lemma~3.17]{Loub}.

\smallskip

Let $h$ be the class number of $\CO_{D}$.
Using the previously mentioned general formula
\[
     R^{}_{\! D} (m)\, := \sum_{i=1}^h R^{}_{q_i}(m) 
     \,=  \sum_{k | m}\chi^{}_{\!\ts D} (k)
\]
for the total (weighted) representation number of a number $m$ by all
forms $q_i$ of discriminant $D$, one can derive explicit results also
in the present case, where $h>1$, but all $h$ forms $q_i$ lie in
different genera. Our previous discussion implies that the supports of
the various $R_{q_i}$ in $(\BZ / D\BZ)^{\times}$ are disjoint and
cover the kernel of $\chi^{}_{\!\ts D}$. Notice that all
representations are counted in this formula, not just the primitive
ones.

\medskip 

The list of the known discriminants of positive definite
binary single class genera is given (without further explanation) in
\cite[Sec.~5.2]{Buell}. Among the discriminants $\equiv 0$ mod $4$,
there are presently $65$ such numbers known, which were already
studied by Gau{\ss} and Euler; see \cite[Sequence~A\ts
000926]{online}.   These numbers are also given in a Table on p.~60 of
\cite{Cox}, sorted according to the class number, which also goes back
to Gau{\ss}. Among the remaining discriminants, namely those $\equiv
1$ mod $4$, further $36$ cases are known \cite[Sec.~5.2]{Buell}. As
before, they contain both fundamental and non-fundamental ones, and
the figures contain the cases of our Tables \ref{tab:max} and
\ref{tab:non-max}.

The total list is believed to be complete, and it has been a long
standing challenge of the `analytic theory of algebraic numbers' to
actually prove this. For a first general approach and a
non-constructive finiteness result (naturally not for today's state of
matters), see the classic lecture notes by Siegel
\cite[Thm.~25.5]{Siegel}. The known list of fundamental discriminants
is complete if the generalised Riemann hypothesis is true
\cite{Loub}. By \cite{Wein}, there is at most one further fundamental
discriminant with only one class in each genus.  The case of arbitrary
discriminants can be reduced relatively easily to the case of
fundamental discriminants, for instance by the method explained in
\cite[Sec.~7.1]{Buell}, or by using the relative class number formula,
as explained in \cite{Sch123}, see also \cite[Excs.~7.3]{Cox}.

\bigskip
When the class number fails to be $1$, we will generally lose
multiplicativity of the counting function $f$. This relates to the
fact that the product of two non-principal ideals in the
corresponding order is principal. However, due to the structure
of the ideal class group, we have a natural binary grading on the
ideals, depending on whether they are principal or not. If the
order under investigation is still principal, one can derive
the generating function quickly from the zeta function.

\begin{example}[${\BZ [\ii\sqrt{6}\,]}$]
The discriminant is $D=-24$, which is fundamental, with
class number $2$, hence ideal class group $C_{2}$. The norm form
is $x^2 + 6 y^2$, which is the norm of principal ideals
in the maximal order $\CO$, while the non-principal ideals
have a norm of the form $2 x^2 + 3 y^2$. The relevant, totally
multiplicative character $\chi^{}_{-24}$ is defined by
\[
   \chi^{}_{-24} (p) = \begin{cases}
   0 ,& \text{if $p=2$ or $p=3$,} \\
   1 ,& \text{if $p\equiv 1,5,7,11$ mod $24$,} \\
  -1 ,& \text{if $p\equiv 13,17,19,23$ mod $24$,}
   \end{cases} 
\]
which leads to the zeta function $\zeta^{}_{K} (s) = \zeta(s) \,
L(s,\chi^{}_{-24})$; compare \cite{Zag}.  Extracting the contribution
from primitive ideals then gives the factorisation
\[
   \zeta^{}_{K} (s) = \zeta (2s)\,
   \biggl( (1+2^{-s})(1+3^{-s})\!\!
   \prod_{p\equiv 5,11\; (24)}
   \frac{1+p^{-s}}{1-p^{-s}} \biggr)
   \prod_{p\equiv 1,7 \; (24)}
   \frac{1+p^{-s}}{1-p^{-s}}.
\]
The bracketed term contains the contributions from primitive
ideals that are themselves not principal, while the last
product covers the principal ones. Our Dirichlet series
thus reads
\[
   D^{\textsf{pr}} (s) = \prod_{p\equiv 1,7 \; (24)}
   \frac{1+p^{-s}}{1-p^{-s}} \;
   \sum_{m=1}^{\infty} \frac{ b(m)} {m^s}
\]
with $b(1)=1$ and $b(m)=0$ whenever $p|m$ for some
$p\equiv 1,7,13,17,19,23$ mod $24$. What remains are the
integers of the form $m=2^{\alpha}\ts 3^{\beta} 
\prod_{p\equiv 5,11 \; (24)} p^{\ell_p}$ with
$\alpha, \beta\in \{0,1\}$ and $\ell_p \in \BN_0$,
only finitely many of them $\ne 0$. For them, the
grading implies
\[
   b(m) = \bigl(1 + (-1)^{\alpha + \beta + \sum \ell_p} 
    \bigr)^{\mathrm{card} \{ p > 3 \ts\mid\ts \ell_p \ne 0 \}},
\]
which, together with the contribution from primes
$\equiv 7$ mod $24$, results in
\[
   D^{\textsf{pr}}(s) =
   1 + \frac{1}{6^s} + \frac{2}{7^s} + \frac{2}{10^s}
   + \frac{2}{15^s} + \frac{2}{22^s} + \frac{2}{25^s}
   + \frac{2}{31^s}  + \frac{2}{33^s} + \frac{2}{42^s}
   + \frac{4}{55^s}  + \ldots
\]
thus illustrating the calculation explained above.
\end{example}

The general situation for non-fundamental discriminants is
more complicated.  To work out further examples, it is advantageous to
start from an order $\CO$ and its $\SOS$-group, which only depends on
the quadratic field $K$ by Theorem~\ref{sos-group}.  Then, for each
element of this group, one has to determine the index of the
corresponding primitive SSL of $\CO$, which can be linked to the
results for the maximal order $\CO_{K}$. Defining the denominator of
$z\in \SOS (\gG)$ for a planar lattice $\gG$ as
\[
       \den^{}_{\gG} (z) = \min \{ \alpha\geqslant 1 \mid \alpha z\gG 
           \subset \gG \} \ts ,
\]
which exists by a standard discreteness argument on the basis of the
lattice property of $\gG$, one sees that $z$ gives rise to a primitive
SSL of $\gG$ of index $ \bigl(\den^{}_{\gG} (z) \bigr)^2$. Since the
latter is an integer, the denominator itself is either an integer or a
quadratic irrationality.

\begin{example}[${\BZ [3\ts\ii]}$ and ${\BZ [5\ts\ii]}$]
\label{ex:3and5}
  Let $p$ be a prime and consider $\BZ [p\ts\ii]$, which is an order
  in the field $\BQ(i)$, with conductor $p$ in the maximal order
  $\BZ[\ii]$. The case $p=2$ was treated in Example~\ref{D=-16} as a
  special case with class number $1$.  Two further primes lead to
  convenient numbers, namely $p=3$ and $p=5$. By
  Theorem~\ref{sos-group}, we have
\[
    \SOS (\BZ[n\ts\ii]) = \SOS (\BZ[\ii])
    \simeq C_{8} \times \BZ^{(\aleph^{}_{0})} ,
\]
for arbitrary $n\in\BN$, with the group and generators as
described in Example~\ref{ex:square-sos}.

To determine the SSLs of $\BZ[p\ts\ii]$, it is again sufficient to
concentrate on the primitive ones, meaning (by
Proposition~\ref{prop:ideals}) the principal ideals of $\BZ[p\ts\ii]$
that are primitive as sublattices.  They can be obtained from the
rotations of the $\SOS$-group (which does not depend on $p$) by
determining the corresponding denominators (which depend on $p$). If
$p$ is prime, the denominator of any $z\in\SOS (\BZ[\ii])$ for
$\BZ[p\ts\ii]$ is either the same as for $\BZ[\ii]$, or it gets
multiplied by $p$. Each primitive SSL of $\BZ[\ii]$, labelled by some
$z = w/\lvert w \rvert$ with $w=m+n\ts \ii$ and $m,n$ coprime, gives
rise to two distinct SSLs of $\BZ[p\ts\ii]$ whose indices might differ
by a factor of $p^2$. This follows from the different point
symmetries, because $z$ and $\ii\ts z$ define the same SSL of the
square lattice, but distinct ones for $\BZ[p\ts\ii]$.  Let us thus
consider Gaussian integers $w=m+\ii \ts n$ with $m,n$ coprime, compare
it with $\ii\ts w$, and distinguish the possible cases.

For $p=3$, a Gaussian integer $w=m + 3 n \ii $ with $3 \!\nmid\! m$
results in $\lvert w\rvert^2 \equiv 1 \bmod{3}$, while $w=m + \ii \ts
n$ with $3 \! \nmid \! n$ gives either $\lvert w\rvert^2 \equiv 1
\bmod{3}$ (when $3 \!\mid\! m$) or $\lvert w\rvert^2 \equiv 2
\bmod{3}$ (when $3\!\ts \nmid\! m$).  Of these possibilities, only
$w=m+ 3 n \ii $ leaves the index of the resulting SSL unchanged (in
comparison to the square lattice), while all other indices have to be
multiplied by $9$.  This gives the generating function
\[
     D^{\sf pr}_{\BZ[3\ts\ii]} (s) \, = \!\!
     \sum_{\substack{m\geqslant 1 \\ m\equiv 1\; (3)}} \!\!
     (1+9^s)\, \frac{f^{\sf pr}_{\square} (m)}{(9\ts m)^s} \; + \!\!
     \sum_{\substack{m\geqslant 1 \\ m\equiv 2\; (3)}} \!\!
     \frac{2\ts f^{\sf pr}_{\square} (m)}{(9\ts m)^s} 
\]
with the arithmetic function $f^{\sf pr}_{\square}$ of
Example~\ref{ex:square-tri}.  There is no meaningful Euler product
expansion, in line with the non-multiplicativity of the total
number of SSLs of a given index in this case.

Similarly, for $p=5$, one has $\lvert w\rvert^2 \equiv \pm 
1 \bmod{5}$ when $w=m + 5 n \ii $ with $5\! \nmid\! m$ or
when $w=m+\ii \ts n$ with $5 \!\mid\! m$ and $5 \!\ts\nmid\! n$, while
$\lvert w\rvert^2 \equiv 0$ or $\pm\ts 2 \bmod{5}$ when
$w=m + \ii \ts n$ with both $m$ and $n$ coprime to $5$. This time,
the generating function reads
\[
     D^{\sf pr}_{\BZ[5\ts\ii]} (s) \, = \!\!
     \sum_{\substack{m\geqslant 1 \\ m\equiv \pm 1\; (5)}} \!\!
     (1+25^s)\, \frac{f^{\sf pr}_{\square} (m)}{(25\ts m)^s} \, + \!\!
     \sum_{\substack{m\geqslant 1 \\ m\equiv 0,\pm 2\; (5)}} \!\!
     \frac{2\ts f^{\sf pr}_{\square} (m)}{(25\ts m)^s} 
\]
with an analogous interpretation as in the previous case.
\end{example}

\section{General case}

Beyond the cases described so far, one loses the possibility
to express the results via simple congruence conditions on
the rational primes. Instead, one needs a criterion for the
representability of a given prime by the norm form via a
specific polynomial congruence, as explained in \cite{Cox}.
When we are dealing with lattices that are similar to the
maximal order in an imaginary quadratic field, we may employ
the main result of Cox \cite{Cox}, as extracted from his
theorems 9.2 and 13.23. It is formulated for discriminants
of the form $-4n$, with class number $h(-4n)$. Its extension
to the remaining discriminants is mentioned in
\cite[Exs.~9.3]{Cox}.
\begin{fact}
   For $n\in \BN$, there exists an effectively computable
   polynomial $f_n (x)$ of degree $h(-4n)$ such that,
   for any odd prime $p$ not dividing $n$, the equation
   $p=x^2 + n \ts y^2$ has an integer solution if and only
   if\/ $\bigl(\frac{-n}{p}\bigr)=1$ and $f_n (x) \equiv 0$
   mod $p$ has an integer solution.

   The corresponding statement also holds for negative discriminants
   $D\equiv 1$ mod $4$, then for the representation of $p$ by the form
   $x^2 + xy + \frac{1-D}{4}\ts y^2$. Here, the conditions are
   $\bigl(\frac{m}{p}\bigr)=1$ with $m=\frac{1-D}{4}$, and the 
   polynomial has degree $h(D)$. \qed
\end{fact}
One possible choice of the polynomial is the class equation, which can
be expressed as a product over the classes and involves the
$j$-invariants of its representatives, see \cite[p.~298]{Cox} for an
example. For fundamental discriminants, there are simpler, more
efficient alternatives for the class polynomials\footnote{For
instance, see 
\texttt{http://www.exp-math.uni-essen.de/zahlentheorie/classpol/class.html}}.

Unfortunately, this approach does not easily seem to lead to closed
expressions as soon as we are beyond the situation with one class per
genus.  As in the second part of the previous chapter, it is thus
usually easier to employ the denominator of a rotation to come to
concrete results. Let us illustrate this with one final example.

\begin{example}[${\BZ[p \ii]}$ with $p$ an odd prime] 
As in Example~\ref{ex:3and5}, we have
\[
   \SOS \bigl(\BZ[p \ii]\bigr) = 
   \SOS \bigl(\BZ[\ii]\bigr) \simeq 
   C_{8} \times \BZ^{(\aleph^{}_{0})} ,
\]
and, in principle, we can proceed as above.  In particular, using the
same conventions for $\omega=m+\ii n$ as above,
$z=\frac{\omega}{\lvert\omega\rvert}$ has denominator $|\omega|$ or
$p\ts |\omega|$, depending on whether $p$ divides $n$ or not. Indeed, $z$
has denominator $p\ts |\omega|$, if $|\omega|^2=m^2+n^2$ is not a
quadratic residue modulo $p$, or if $|\omega|^2$ is divisible by $p$.
If $|\omega|^2$ is a quadratic residue, both denominators may
occur. Clearly, if $z$ has denominator $|\omega|$, then $\ii z$ has
denominator $p\ts |\omega|$, since $m$ and $n$ are relatively
prime. Hence, for fixed $|\omega|^2$, the number of primitive SSLs
with index $|\omega|^2$ is at most the number of primitive SSLs with
index $p^2\ts |\omega|^2$. Thus, in terms of the arithmetic function
$f^{\mathsf{pr}}_{\square}$ of the square lattice, we may write
\[
   D_{\BZ[p \ii]}^{\mathsf{pr}}(s) = \sum_{\leg{m}{p}=1}
    \frac{f_\square^{\mathsf{pr}}(m)}{p^{2s} m^s}
    \left(  b(m)+a(m)\ts p^{2s} \right)+
    \sum_{\leg{m}{p}\neq 1}\frac{2f_\square^{\mathsf{pr}}(m)}{p^{2s} m^s}, 
\]
where $a(m)$ and $b(m)$ are still to be determined. They satisfy
$a(m)+b(m)=2$ together with $a(m)\leq b(m)$ and
$f_\square^{\mathsf{pr}}(m)\ts a(m)\in\BN_0$ (we have seen above
that $a(m)=b(m)=1$ for $p=3$ or $p=5$).

The determination of $a(m)$ depends on the prime factorisation of $m$
and is rather tedious in general. As an example, we discuss $p=7$,
where the quadratic residues are $1$, $2$ and $4$.  Here, we have three
different types of prime numbers $q=(m+\ii n)(m-\ii n)$ (we only need
to consider primes $q\equiv 1\pmod 4$), namely
\begin{enumerate}
\item $q \equiv 1,2,4 \pmod 7$ and either $7\mid m$ or $7\mid n$
\item $q \equiv 1,2,4 \pmod 7$ and $7\nmid m$, $7\nmid n$, which implies
  $m^2\equiv n^2 \pmod 7$,
\item $q \equiv 3,5,6 \pmod 7$, which implies $7\nmid m$, $7\nmid n$,
  and $m^2\not\equiv n^2 \pmod 7$.
\end{enumerate}
Note that $a(q)=1$ in the first case, and $a(q)=0$ in the other two.
To handle composite numbers $m$, let $m=\prod_i q_i^{r_i}$ be the
prime decomposition of $m$ and define $s(m):=\sum_i t_i r_i$, where
$t_i$ is $4$, $2$ or $1$, according to whether $q_i$ is of type $1$,
$2$ or $3$, respectively. One can check that $s(m)$ is even if and
only if $m \equiv 1,2,4 \pmod 7$. Such numbers $m$ can be divided into
three equivalence classes, namely
\begin{itemize}
\item[$N_1$:] all prime factors
   $q_i \equiv 3,5,6 \pmod 7$ have even power and $s(m) \equiv 0 \pmod 4$,
\item[$N_2$:]  all prime factors
   $q_i \equiv 3,5,6 \pmod 7$ have even power and $s(m) \equiv 2 \pmod 4$,
\item[$N_3$:]  there are
   at least two prime factors $q_i \equiv 3,5,6 \pmod 7$ with odd power.
\end{itemize}
Both $a(m)$ and $b(m)$ are constant on these equivalence classes, with
values
\begin{center}
{\renewcommand{\arraystretch}{1.1}
\begin{tabular}{c|ccc}
 $m$ & $N_1$ & $N_2$ & $N_3$ \\
\hline
$a(m)$ & $1$ & $0$ & $\frac{1}{2}$ \\
$b(m)$ & $1$ & $2$ & $\frac{3}{2}$
\end{tabular}}
\end{center}
This gives the generating function
\[
\begin{split}
   D_{\BZ[7\ii]}^{\mathsf{pr}}(s)\, = & \sum_{m\in N_1}
   \frac{f_\square^{\mathsf{pr}}(m)}{49^s m^s}\left(1+49^s\right)
   \, + \! \sum_{m\in N_2}\frac{2f_\square^{\mathsf{pr}}(m)}{49^s m^s} \\
   & + \sum_{m\in N_3}\frac{f_\square^{\mathsf{pr}}(m)}
      {49^s m^s}\left(\frac{3}{2}+
   \frac{49^s}{2}\right) \, + \!\! \sum_{m\equiv 3,5,6\, (7)}
   \frac{2f_\square^{\mathsf{pr}}(m)}{49^s m^s} \\
   = &\; \frac{2}{49^s} D_{\BZ[\ii]}^{\mathsf{pr}}(s) 
     \, + \! \sum_{m\in N_1}
   \frac{f_\square^{\mathsf{pr}}(m)}{49^s m^s}\left(49^s-1\right)
   \, + \! \sum_{m\in N_3}\frac{f_\square^{\mathsf{pr}}(m)}{49^s m^s}
   \left(\frac{49^s}{2}- \frac{1}{2}\right) \\
   = &\; 1+\frac{1}{49^s}+\frac{2}{50^s}+\frac{2}{53^s}+\frac{2}{58^s}
    +\frac{2}{65^s}+\frac{2}{74^s}+\frac{2}{85^s}+\frac{2}{98^s}
    +\frac{2}{113^s}+ \frac{2}{130^s}+ \ldots
\end{split}
\]
which illustrates the higher complexity of this case.
\end{example}

\bigskip
\section*{Acknowledgements}

It is our pleasure to thank Robert V.\ Moody for helpful discussions
and Christian Huck for various comments on the manuscript.  This work
was supported by the German Research Council (DFG), within the CRC
701.


\bigskip

\begin{thebibliography}{99}
\small

\bibitem{Apo-2}
T.\thinspace M.~Apostol.
\textit{Modular Functions and Dirichlet Series in Number Theory},
2nd ed., Springer, New York (1990).

\bibitem{B}
M.~Baake,
Solution of the coincidence problem in dimensions $d\le 4$,
in:\ \textit{The Mathematics of Long-Range Aperiodic Order},
ed.\ R.\thinspace V.\ Moody, NATO-ASI C 489, Kluwer,
Dordrecht (1997), pp.\ 9--44;  rev.\ version, 
\texttt{arXiv:math.MG/0605222}.

\bibitem{BG}
M.~Baake and U.~Grimm.
Bravais colourings of planar modules with $N$-fold symmetry, 
\textit{Z.\ Kristallographie} {\bf 219} (2004) 72--80;
\texttt{arXiv:math.CO/0301021}.

\bibitem{BHGZ}
M.~Baake, M.~Heuer, U.~Grimm and P.~Zeiner,
Coincidence rotations of the root lattice $A_4$,
\textit{European J.\ Combinatorics} {\bf 29} (2008) 1808--1819;
\texttt{arXiv:0709.1341}.

\bibitem{BHM}
M.~Baake, M.~Heuer and R.\thinspace V.~Moody,
Similar sublattices of the root lattice $A_4$, 
\textit{J.\ Algebra} {\bf 320} (2008) 1391--1408;
\texttt{arXiv:math.MG/0702448}.

\bibitem{BM1}
M.~Baake and R.\thinspace V.~Moody, 
Similarity submodules and semigroups,
in: {\em Quasicrystals and Discrete Geometry}, 
ed.\ J.\ Patera, Fields Institute Monographs, vol.~10, 
AMS, Providence, RI (1998), pp.\ 1--13.

\bibitem{BM2}
M.~Baake and R.\thinspace V.~Moody, 
Similarity submodules and root systems in four dimensions, 
\textit{Can.\ J.\ Math.} {\bf 51} (1999) 1258--1276;
\texttt{arXiv:math.MG/9904028}.

\bibitem{BS}
S.\thinspace I.~Borewicz and I.\thinspace R.~Safarevic,
\textit{Zahlentheorie}, edited by H.\ Koch,
Birkh\"auser, Basel (1966).

\bibitem{Buell}
D.\thinspace A.~Buell,
\textit{Binary Quadratic Forms -- Classical Theory
and Modern Computations}, Springer, New York (1989).



\bibitem{CRS}
J.\thinspace H.~Conway, E.\thinspace M.~Rains and 
N.\thinspace J.\thinspace A.~Sloane,
On the existence of similar sublattices,
\textit{Can.\ J.\ Math.} {\bf 51} (1999) 1300--1306.

\bibitem{CS}
J.\thinspace H.~Conway and N.\thinspace J.\thinspace A.~Sloane,
\textit{Sphere Packings, Lattices and Groups}, 3rd ed.,
Springer, New York (1999).

\bibitem{Cox}
D.\thinspace A.~Cox,
\textit{Primes of the Form $x^2 + n\ts y^2$},
corr.\ printing, Wiley, New York (1997).

\bibitem{GB}
S.~Glied and M.~Baake,
Similarity versus coincidence rotations of lattices,
\textit{Z.\ Krist.} {\bf 223} (2008) 770--772;
\texttt{arXiv:0808.0109}.

\bibitem{G}
S.~Glied,
Similarity and coincidence isometries for modules,
\textit{Can.\ Math.\ Bulletin}, in press.

\bibitem{HW}
 G.\thinspace H.~Hardy and E.\thinspace M.~Wright,
\textit{An Introduction to the Theory of Numbers},
6th ed., revised by D.\thinspace R.\ Heath-Brown and 
J.\thinspace H.\ Silverman,
Oxford University Press, Oxford (2008).

\bibitem{H}
C.~Huck,
A note on the coincidence isometries of modules
in Euclidean space,
\textit{Z.\ Krist.} {\bf 224} (2009) 341--344;
\texttt{arXiv:0811.3551}.

\bibitem{Loub}
S.~Louboutin, 
Minorations (sous l'hypoth\`ese de Riemann g\'en\'eralis\'ee) des
nombres de classes des corps quadratique
imaginaires. Application. 
\textit{C.\ R.\ Acad.\ Sci.\ Paris Ser.\ I}
\textbf{310} (1990) 795--800.

\bibitem{Marcus}
D.\thinspace A.~Marcus,
\textit{Number Fields},
Springer, Berlin (1977).

\bibitem{PBR}
P.\thinspace A.\thinspace B.~Pleasants, M.~Baake and J.~Roth,
Planar coincidences with $N$-fold symmetry,
\textit{J.\ Math.\ Phys.} {\bf 37} (1996) 1029--1058;
corr.\ version \texttt{arXiv:math.MG/0511147}.

\bibitem{Sch123}
R.\ Scharlau,
\textit{Seminar \"uber komplexe Multiplikation},
parts 1, 2, and 3, available online at
\texttt{http://www.mathematik.uni-dortmund.de/\~{}scharlau/research/}

\bibitem{S}
J.-P.~Serre,
\textit{A Course in Arithmetic},
4th corr.\ printing, Springer, New York (1993).

\bibitem{Siegel}
C.\thinspace L.~Siegel,
\textit{Analytische Zahlentheorie II},
lecture notes, edited by K.\thinspace F.\ K\"{u}rten
and G.\ K\"{o}hler, Univ.\ G\"ottingen (1964).

\bibitem{online}
N.\thinspace J.\thinspace A.~Sloane,
\textit{The Online Encyclopedia of Integer Sequences},
\newline
\texttt{http://www.research.att.com/\~{}njas/sequences/}

\bibitem{T}
G.~Tenenbaum,
\textit{Introduction to Analytic and Probabilistic Number Theory},
CUP, Cambridge (1995).

\bibitem{Wein}
P.~Weinberger,  
Exponents of class groups of quadratic fields,
\textit{Acta Arithm.} \textbf{22} (1973) 117--124.

\bibitem{Zag}
D.\thinspace B.\ Zagier,
\textit{Zetafunktionen und quadratische K\"orper},
Springer, Berlin (1984).

\bibitem{Z}
P.~Zeiner,
Symmetries of coincidence site lattices of cubic lattices,
\textit{Z.\ Krist.} {\bf 220} (2005) 915--925;
\texttt{arXiv:math/0605525}.

\end{thebibliography}
\end{document}